\documentclass[12pt]{article}

\usepackage{amsfonts}
\usepackage{graphics}
\usepackage{amssymb}

\newtheorem{theorem}{Theorem}[section]

\newtheorem{definition}{Definition}[section]
\newtheorem{proposition}[theorem]{Proposition}
\newtheorem{corollary}[theorem]{Corollary}

\newtheorem{remark}{Remark}[section]

\begin{document}

\title{Generalized solutions of the multidimensional stochastic Burgers equation.}

\author{  P. Catuogno\footnote{Departamento de Matem\'{a}tica, Universidade Estadual de Campinas, Brazil. 
E-mail:  {\sl  pedrojc@ime.unicamp.br}.
},
J.F. Colombeau \footnote{ Institut Fourier, Universit\'e de Grenoble., France. 
E-mail: {jf.colombeau@wanadoo.fr}.
}, 
C. Olivera\footnote{Departamento de Matem\'{a}tica, Universidade Estadual de Campinas, Brazil. 
E-mail:  {\sl  colivera@ime.unicamp.br}.
}}

\date{}

\maketitle


\vspace{0.3cm} \noindent {\bf MSC2000 subject classification:} 60H15, 35R60, 46F99.

\vspace{0.3cm}

\begin{center}
\begin{abstract}

\noindent   We introduce a new concepts of weak solution for the 
conservative  stochastic Burgers equation in any dimension. The definition is based on  weak solution concepts   introduced by various authors in order to make sense of  equations which  do not  solutions in the sense of distributions. In one dimension the solution reduces to the classical distributional solution of the 1--D stochastic Burgers equation.
\end{abstract}
\end{center}

\maketitle

\section {Introduction}

\noindent

The aim of this paper is to study the existence of solution to the multidimensional
stochastic Burgers equation in $\mathbb{R}^d$ of the following form:

\begin{equation}\label{burg}
 \left \{
\begin{array}{lll}
    \partial_t U(t, x) =  \ \Delta_x U(t, x) +    \ \nabla_x   \|U(t,x)\|^{2} +   \nabla_x  \dot{W}(t,x),\\
 U(0, x) = \nabla_x  f(x).
\end{array}
\right .
\end{equation}

\noindent where $\dot{W}(t,x)$ is a space--time white noise.\\

\noindent During the past few decades, the stochastic Burgers equation has found applications in
diverse fields ranging from statistical physics, cosmology to fluid dynamics. The problem of 
Burgers turbulence, that is the study of the solutions of the Burgers equation with random
initial conditions or random forcing is a central issue in the study of nonlinear systems out of equilibrium. See J. Bec, K. Khanin \cite{Beck} and E. Weinan \cite{Wie}.\\

\noindent The multidimensional conservative stochastic Burgers equation has been studied by several authors in the case  
$\dot{W}(t,x)$ is a noise white in time 
and correlated in space. See A. Dermoune \cite{Dermoune}, 
Z. Brezniak, B. Goldys, M. Neklyudov \cite{BGN} and R. Iturriaga,  K. Khanin \cite{IK}. In the case that the 
nonlinear term is interpreted in the sense of Wick calculus the stochastic Burgers equation has been well studied, see S. Assing \cite{AS} and also H. Holden, T. Lindstr\o m, B. \O ksendal, J. Ub\o e, T. Zhang \cite{HLOUZ1} and \cite{HLOUZ2}. They showed existence and uniqueness results for the solution regarded as a stochastic process with values in a Kondratiev space of stochastic distributions. \\

\noindent Up to our knowledge there is no proof of existence of solution in any sense    
for equation (\ref{burg}) when $d \geq 2$. In the case that $d=1$, the problem of existence of solutions for stochastic Burgers equation is well understood, see L. Bertini, N. Cancrini, G. Jona--Lasinio \cite{BCJL} , G. Da Prato, A. Debussche, R. Teman \cite{dPDT},   P. Catuogno, 
C. Olivera \cite{CO}  and     M. Gubinelli, N.Perkowski  \cite{Gubi}.

\noindent A main difficulty with the multidimensional stochastic Burgers equation is that the solutions 
take values in a distributional space. Therefore, it is necessary to give meaning to the 
non--linear term $\nabla_x   \|U(t,x)\|^{2}$. In the one--dimensional case the Cole--Hopf 
transformation is used to sort  out this problem, but this  does not make sense 
in the  multidimensional case  since the solution of the  stochastic heat equation
is not a standard stochastic process \cite{Walsh}. \\

\noindent We remark that it is believed that the Cole--Hopf solution  $U(t, x):=\nabla_x \log Z(t,x)$, where $Z(t,x)$ denotes the solution of the stochastic  heat equation with multiplicative noise, is the correct physical solution of (\ref{burg}), see  \cite{BCJL} and \cite{BG}.
 We mention also that  the Cole--Hopf solution is related to the KPZ equation, see \cite{Hairer}.  
 
\noindent  We introduce a new concept of solution in  order to give sense to this expected  solution of the multidimensional stochastic Burgers equation. This new concept of solution is strongly inspired by previous 
works of many  authors using generalized solutions for equations of mathematical physics in cases in that the expected or known (for instance from numerical investigations) solutions do not make sense within distribution theory: \cite{  colb, colb2,  Grosser, ober2, Pa}. In particular such generalized solutions have been considered for stochastic equations in \cite{albe},  \cite{Russo}, \cite{ober4}. \\

\noindent We denote by $\mathcal{D}((0,T) \times \mathbb{R}^d ; \mathbb{R}^d)$ the space of
the infinitely differentiable functions with
compact support in   $(0,T) \times \mathbb{R}^d$ and values in $\mathbb{R}^d$. 

\noindent Let us denote by $\mathcal{L}_B$ the Burgers operator, 
\[
\mathcal{L}_B\varphi = \partial_t \varphi -  \ \Delta_x \varphi -    \ \nabla_x   \|\varphi\|^{2} 
\]
for $\varphi \in \mathcal{D}((0,T) \times \mathbb{R}^d ; \mathbb{R}^d)$.

\noindent Now, we introduce the concept of weak solution for the stochastic Burgers equation. 

\begin{definition}\label{solu3}
We say that a sequence  $(U_n)$ of smooth semimartingales 
is an $\mathbb{R}^d$--{\it weak generalized solution} of the equation  (\ref{burg})
if
\begin{enumerate}
\item For each $\varphi \in \mathcal{D}((0,T) \times  \mathbb{R}^d ;  \mathbb{R}^d)$,
\[
\lim_{n \rightarrow \infty}( \mathcal{L}_BU_n, \varphi ) =  ( \nabla \dot{W}, \varphi)_{\mathcal{D^{\prime}}}, 
\]

where the convergence is in $L^{2}(\Omega)$.

\item $U_n (0,x)=\nabla f(x)$.
\end{enumerate}

\end{definition}

\noindent  The above means that, after averaging on any smooth  test function, the sequence $(U_n)$  tends to satisfy the equation when $n\rightarrow \infty$. In general we ignore the nature of the possible limit of the sequence $(U_n)$: this limit can be a distribution which is not necessarily a solution of the equation in the sense of  distribution theory, such as in the case of shock waves for nonconservative systems \cite{Ara,colb2}, in the case of  delta waves in general and a fortiori delta--prime wave \cite{Pa}, or it can be an object which is not a distribution, such as  objects put in evidence in general relativity in \cite{Stein}. In one  dimension the limit of the sequence $(U_n)$ that we construct  exists as a distribution which  is a  solution in the sense of distributions. \\

\noindent  We  shall show  that there exists a weak generalized solution for the multi--dimensional stochastic Burgers equation, which is  a  type of Cole--Hopf  solution.

This paper is organized as follows. In the next section we review some facts on space--time white noise. In section 3 we show the 
existence of such a weak solution for the stochastic Burgers equation (\ref{burg}). Then we  recover the stronger result
of existence of a distributional solution for the stochastic Burgers equation in dimension $1$.

\section{Preliminaries}

\noindent We say that a random field $\{ S(t,x) : t \in [0,T],~x \in \mathbb{R}^d \}$ is a spatially dependent semimartingale if for
each $x\in \mathbb{R}^d$,
$\{S(t,x): t \in [0,T]\}$ is a $\mathbb{R}^d$--valued semimartingale in relation to the same filtration $\{ \mathcal{F}_t : t \in [0,T] \}$. If
$S(t,x)$ is a $C^{\infty}$--function of $x$ and continuous in $t$  almost surely (i.e., for almost all $\omega$),  it is called  a smooth semimartingale.  See \cite{Ku} for a rigorous study of spatially depend semimartingales
and applications to  stochastic differential equations.

\noindent A distribution valued Gaussian process with mean zero $\{\dot{W}(t,x): t \in [0,T] ~, ~x \in \mathbb{R}^d\}$ is a
space--time white noise if
\[
 \mathbb{E}( \dot{W}(t,x) \dot{W}(s,y) ) = \delta(t-s) \delta(x-y).
\]
More precisely :

\begin{itemize}

 \item For any $\xi \in \mathcal{D}((0,T) \times \mathbb{R}^d)$ the random variable $\dot{W}(\xi)$
 is  Gaussian variable  with mean zero. 

\item For any $\xi, \rho \in  \mathcal{D}((0,T) \times \mathbb{R}^d)$ the random variables
$\dot{W}(\xi), \dot{W}(\rho)$  have  covariance 
\[
 \mathbb{E}(\dot{W}(\xi) \dot{W}(\rho))=
\int_{[0,T]\times \mathbb{R}^d} \xi(t,x) \cdot \rho(t,x) dx dt.
\]

\end{itemize}

\noindent It is not difficult to construct a space time--white noise. In fact, let $\{f_j : j \in \mathbb{N} \}$ be a complete orthonormal basis of 
$L^{2}([0,T] \times \mathbb{R}^d)$ and $\{ Z_j : j \in \mathbb{N} \}$ 
be a family of independent Gaussian random 
variables with mean zero and variance one. Then 
\begin{equation}\label{eq1wn}
 \dot{W}(t,x) := \sum_{j=1}^{\infty} f_j(t,x)Z_j
\end{equation}
is a  space--time white noise, where the action is defined as 
\[
 (\dot{W}, \xi  )_{\mathcal{D}^{\prime}} = \sum_{j=1}^{\infty} ( \xi, f_j) Z_j.
\]

It is well know that the last  action can be extended to  $\xi \in L^{2}([0,T]\times \mathbb{R}^{d})$ using It\^o isometry.

\noindent The cylindrical Wiener process $\{ W_t : L^2(\mathbb{R}^d) \rightarrow L^2(\Omega) : t\in [0,T] \}$ associated to $\dot{W}$ is 
given by 
\[
 W_t (\varphi):= (\dot{W}, \varphi 1_{[0,t]}).
\]
It is clear that $ W_t (\varphi)$ is a Brownian motion with variance $t\|\varphi\|^2$ for each $\varphi \in L^2(\mathbb{R}^d)$. 

\noindent We say that a sequence of smooth semimartingales $(W^n)$ is a {\it weak approximation} to the cylindrical Wiener $W$ if for all 
$\varphi \in L^2(\mathbb{R}^d)$,
\begin{equation}\label{conap}
 \lim_{n \rightarrow \infty} \int_{\mathbb{R}^d} \varphi(x)  W^n_t(x) dx =W_t(\varphi)  ,
\end{equation}

\noindent  where the convergence is in  $C([0,T], L^{2}(\Omega))$.

\noindent A weak approximation $( W^n)$ to the cylindrical Wiener process $W$ is {\it good} if for each $n \in \mathbb{N}$, $W^n_t(x)$ is a Brownian motion with  quadratic variation 
\[
\langle W^n( x) \rangle_t =C_n \cdot t,
\] 
 for all $x\in \mathbb{R}^d$ and where $C_n$ is a constant depending on $n$.  \\ 
  

\noindent A main ingredient in our approach for solving the Burgers equation is the use of regularization techniques for the 
white noise with respect to  space, we refer the reader 
to \cite{BG} and \cite{ober2} to background material.  

\noindent Let $\rho :  \mathbb{R}^d \rightarrow [0, \infty)$ be an infinitely differentiable symmetric function with
compact support such that $\int_{\mathbb{R}^d} \rho(x) \ dx=1$. We will consider the mollifiers $\rho_n (x)= n^d \rho(nx)$, with $n \in \mathbb{N}$. 
The regularization by $\rho$ of the space--time white noise $\dot{W}$, denoted by $\dot{W}_{\rho_n}$, are defined to be 
\[
 \dot{W}_{\rho_n}(t,x) := \rho_n \ast \dot{W}(t,x).
\]
We observe that $\dot{W}_{\rho_n}$ is white in time and colored in space, in fact we have that $\dot{W}_{\rho_n}(t,x)$ 
is a distribution valued Gaussian process with mean zero and covariance,

\[
 \mathbb{E}( \dot{W}_{\rho_n}(t,x) \dot{W}_{\rho_n}(s,y) ) = \delta(t-s) h_n(x-y)
\]
where $h_n : \mathbb{R}^d \rightarrow \mathbb{R}$ is given by $h_n(z)= \int_{\mathbb{R}^d}\rho_n (u)\rho_n (u+z)du$.

\noindent In terms of the expansion (\ref{eq1wn}) we have that 
\[
 \dot{W}_{\rho_n}(t,x) := \sum_{j=1}^{\infty} \rho_n \ast f_j(t,x)Z_j . 
\]

\noindent The mollified cylindrical Wiener process $W_{t}^{\rho_n}(x)$ associated with the space--time white noise $\dot{W}(t,x)$ is defined by
\begin{equation}
W_{t}^{\rho_n}(x):= W_{t}(\rho_n(x-\cdot)).
\end{equation}
The distributional time derivative of $W_{t}^{\rho_n}(x)$ is $\dot{W}_{\rho_n}(t,x)$.

\noindent We have that $W_{t}^{\rho_n}(x)$ is a Brownian motion with quadratic variation,
\begin{equation}\label{qv}
\langle W^{\rho_n}(x) \rangle_t = \| \rho \|^2n^d \cdot t.
\end{equation}
In particular we have proved the existence of good weak approximations to cylindrical Wiener processes. 
\begin{proposition}
$( W^{\rho_n})$ is a good weak approximation to the cylindrical Wiener process $W$.
\end{proposition}

\section{Existence of weak  solution in any dimension}
Let $( W^n)$ be a good weak approximation to the cylindrical Wiener $W$. We will denote by $H_{n}(t,x)$ the process $\log   Z_{n}(t,x)$, 
where $Z_n$ is the solution of the regularized stochastic heat equation in
the It\^o sense
\begin{equation}\label{heat2}
 \left \{
\begin{array}{lll}
dZ_{n}  & = & \Delta_x Z_{n} \ dt  + Z_{n} \ dW^n, \  \\
Z_n(0,x) & = &  e^{f(x)}.
\end{array}
\right .
\end{equation}
\noindent We observe that $W^{n}(t,x)$ is $C^{\infty}$--martingale then we follow the existence and uniqueness 
of solution of the classical theory,  see for instance \cite{Ku}. Applying the Feynman--Kac formula we obtain the following representation of the solution to equation (\ref{heat2}) with 
$f \in C_b^{\infty}(\mathbb{R}^d)$,
\begin{equation}
 Z_n(t,x)=\tilde{\mathbb{E}}( {\rm e}^{f(x+B_t)} {\rm e}^{\int_0^t W^n(t-s,x+B_s)ds})
\end{equation}
where $B_t$ is a Brownian motion independent of $W^n$ such that $B_0=0$ defined in an auxiliary probability space, see  \cite{Dermoune}.

\noindent The sequence of smooth semimartingales $(U_n)$ given by $U_n := \nabla_x H_n$ is called the {\it Cole--Hopf sequence associated to} 
$(W^n)$. 
\begin{theorem}\label{teoSolu2} Let  $f\in  C_{b}^{\infty}(\mathbb{R}^d)$. The Cole--Hopf sequence associated to $( W^n)$ is a weak generalized solution of 
the stochastic Burgers equation (\ref{burg}).
\end{theorem}

\begin{proof}  Let $(U_n)$ be the Cole--Hopf sequence given by $U_n(t,x) := \nabla_x H_n (t,x)= \nabla_x \log Z_n(t,x)$. Since $Z_n$ 
satisfies the equation (\ref{heat2}) it is a $C^{\infty}$--semimartingale. Then applying  the classical It\^o formula we have that
\begin{eqnarray*}
 dH_n & = & \frac{dZ_n}{Z_n} -\frac{d \langle Z_n \rangle}{2Z_n^2} \\
 &= & \frac{\Delta_x Z_n dt + Z_n dW^{n}}{Z_n} -\frac{Z_n^2 C_ndt}{2Z_n^2} \\
& =& \frac{\Delta_x Z_n dt}{Z_n}+  dW^{n} -\frac{C_n dt}{2}.
\end{eqnarray*}
It is easy to check that 
\[
 \frac{\Delta_x Z_n }{Z_n} = \Delta_x H_n + \| \nabla_x H_n \|^2.
\]
\noindent Combining the above equations we obtain
\begin{equation}\label{eq KPZ}
 dH_n = (\Delta_x H_n + \| \nabla_x H_n \|^2)dt +  dW^{n} -\frac{C_n dt}{2},
\end{equation}
that is
\[
 H_n(t,x)= f(x) +\int_0^t (\Delta_x H_n(s,x) + \| \nabla_x H_n(s,x) \|^2)ds + W^{n}_t(x) -\frac{C_n t}{2}.
\]

\noindent Let $\varphi=(\varphi_1, \cdots , \varphi_d ) \in \mathcal{D}((0,T) \times  \mathbb{R}^d ; \mathbb{R}^d)$. 
Multiplying (\ref{eq KPZ}) by $\partial_t \partial _{x_i}\varphi_i(t,x)$ and integrating on $[0,T]\times \mathbb{R}^d$, we obtain that
\begin{eqnarray*}
\int_{[0,T]\times \mathbb{R}^d} H_n~\partial_t \partial _{x_i}\varphi_i dx dt & = & -\int_{[0,T]\times \mathbb{R}^d} 
(\Delta_x H_n + \| \nabla_x H_n(s,x) \|^2 )\partial _{x_i}\varphi_idx dt \\
& & - \int_{[0,T]\times \mathbb{R}^d} \partial_{t} \partial _{x_i}\varphi_i(t,x)  dW^n_t(x)dx.
\end{eqnarray*}
Thus 
\begin{equation}\label{BR}
( \mathcal{L}_B U_{n}, \varphi ) =  \int_{[0,T]\times \mathbb{R}^d} \nabla_x  \partial_{t} \varphi(t, x)    dW^n_t(x)dx
\end{equation}

Then we conclude 

\begin{equation}\label{C1}
 \lim_{n \rightarrow \infty}( \mathcal{L}_B U_{n}, \varphi ) = (\nabla \dot{W}, \varphi)_{\mathcal{D^{\prime}}},
   \end{equation}
	
	where the convergence is in  $L^{2}(\Omega)$.
$\square$

\end{proof}

\begin{remark} We observe that  by definition of the solution we have 

\[
\lim_{n \rightarrow \infty}( \mathcal{L}_BU_n, \varphi ) =  (\nabla \dot{W}, \varphi)_{\mathcal{D^{\prime}}} \  in \ L^{2}(\Omega).
\]
 
  Via classical arguments, taking a dense and countable set of $\mathcal{D}$,  we  can conclude that  there exists a subsequence of $U_n$
such that

\[
\lim_{n \rightarrow \infty}( \mathcal{L}_BU_n, \varphi ) =  (\nabla \dot{W}, \varphi)_{\mathcal{D^{\prime}}} 
\]

 where the convergence is almost surely in  $\Omega$. 
\end{remark}

\subsection{ Classical distributional solution in the one dimensional case.} We observe that in the one--dimensional case the sequence $( Z_{n})$ of solutions of the regularized stochastic heat equations (\ref{heat2}) 
converges uniformly on compacts of $(0,T) \times \mathbb{R}$ to $Z$, where 
$Z$ is the solution of the stochastic heat equation in the It\^o sense
\begin{equation}\label{heat3}
 \left \{
\begin{array}{lll}
dZ  & = & \Delta_x Z \ dt  + Z \ dW, \  \\
Z(0,x) & = &  e^{f(x)}.
\end{array}
\right .
\end{equation}
See  L. Bertini and N. Cancrini \cite{BC}, Theorem 2.2.

\noindent We will denote by $U(t,x)$ the gradient in the sense of  distributions of $\log Z(t,x)$, that is, $U(t,x)=\nabla_x \log Z(t,x)$.

\begin{corollary}
 $U$ is a distributional solution of the one--dimensional stochastic Burgers equation (\ref{burg}). That is, $U$ verifies
\begin{enumerate}
\item There exists a sequence of $C^{\infty}$--semimartingales  $(U_n)$ 
such that  
\[
U=\lim_{n \rightarrow \infty}U_n 
\]
and there exists $\nabla_x \|U \|^{2}$ such that   
\[
\nabla_x \|U \|^{2}:=\lim_{n\rightarrow\infty} \nabla_x \|U_n\|^{2}  {\mbox{ in }} \mathcal{D}^{\prime}((0,T) \times  \mathbb{R}).
\]

\item For all $\varphi \in \mathcal{D}((0,T) \times  \mathbb{R} ; \mathbb{R})$,
\[
( \mathcal{L}_B U, \varphi ) = (\nabla \dot{W},\varphi).
\]

\item $\nabla_x f$ is the section of $U$ at $t=0$ in the sense of S. \L{}ojasiewicz, see \cite{ober2}.
\end{enumerate}
 
\end{corollary}

\begin{proof}
It is clear that, 
\begin{equation}\label{e1}
 \lim_{n \rightarrow \infty} (U_n , \partial_t\varphi)=(U, \partial_t\varphi)
\end{equation}
and
\begin{equation}\label{e2}
 \lim_{n \rightarrow \infty}(U_n , \Delta_x\varphi)=(U, \Delta_x\varphi).
\end{equation}
\noindent We have that $ \int_{[0,T]\times \mathbb{R}}\varphi(t, x)    dW_t(x)dx$ defines a continuous linear functional from 
$\mathcal{D}((0,T) \times \mathbb{R}; \mathbb{R})$ to $\mathbb{R}$ (see  K. Schauml\"{o}ffel \cite{Sch}) and
\begin{equation}\label{e3}
\lim_{n \rightarrow \infty}   \int_{[0,T]\times \mathbb{R}} \nabla_x  \partial_{t} \varphi(t, x)    dW^n_t(x)dx= (\nabla\dot{W}, \varphi).
\end{equation}

\noindent  From the equation (\ref{BR})
and the  convergences   (\ref{e1}), (\ref{e2}) and (\ref{e3}) we have that for all $ \varphi\in \mathcal{D}((0,T) \times \mathbb{R}; \mathbb{R})$,

\[
 \int_{[0,T]\times \mathbb{R}}  \nabla_x   \|U_n(t,x)\|^{2}  \varphi(t,x)  \ dtdx 
\]
converges and defines a linear functional. Thus the  nonlinearity
\[
(\nabla_x   \|U(t,x)\|^{2} , \varphi) := \lim_{n \rightarrow \infty} \int_{[0,T]\times \mathbb{R}}  \nabla_x   \|U_n(t,x)\|^{2}  \varphi(t,x)  \ dtdx
\]
is well defined.

\noindent Let $\varphi \in \mathcal{D}( \mathbb{R})$ and $\{ \rho_{\varepsilon}: \varepsilon >0 \}$ be a strict delta--net. From the  continuity 
of $ Z(t,x) $,

\begin{eqnarray*}
\lim_{\varepsilon\rightarrow 0} \int_{[0,T]\times \mathbb{R}} \nabla_x \ln \  Z(t,x) \  \rho_{\varepsilon}(t) \   \varphi(x) \
 dt\ dx  & = & 
-  \int_{\mathbb{R}} f(x) \  \nabla_x \varphi(x)\ dx \\
& = & \int_{\mathbb{R}}   \nabla_x  f(x) \   \varphi(x)\ dx.    
\end{eqnarray*}
Thus $\nabla_xf$ is the section of $U$ at $t=0$. $\square$
\end{proof}

\begin{remark} More details on the results of this subsection can be found in \cite{CO}. 
\end{remark}
\section*{Acknowledgements}

The authors P. Catuogno  and C, Olivera are  partially supported by
CNPq through the grant 460713/2014--0. C. Olivera
 also by FAPESP through the grants  2015/04723--2 and 2015/07278--0.
 J.F. Colombeau is supported by  
FAPESP  through the grant 2012/18940--7 .

\end{document}